\documentclass{amsart}

\usepackage{etex}
\reserveinserts{12}%
\usepackage{cmap}
\usepackage[utf8]{inputenc}
\usepackage[T1]{fontenc}
\usepackage[english]{babel}
\usepackage[babel]{csquotes}
\usepackage{microtype,booktabs}
\usepackage{graphicx}
\usepackage{xspace}
\usepackage{xcolor}
\usepackage{multirow}
\usepackage{lmodern}
\usepackage{amsmath,amssymb,amsfonts,amsthm,shuffle}
\usepackage{enumerate}
\usepackage[breaklinks=true]{hyperref}

\DeclareMathOperator{\Span}{span}
\DeclareMathOperator{\rank}{rank}

\newcommand{\Prob}{\mathbb{P}}
\newcommand{\prob}[1]{\mathbb{P}\left(#1\right)}

\newcommand{\indic}[1]{\mathrm{\textbf{1}}_{#1}}
\newcommand{\as}{\text{a.s.}}

\newcommand{\Esp}[1]{\mathbb{E}\left[#1\right]}
\newcommand{\Espc}[2]{\mathbb{E}\left[ #1 \middle\vert #2 \right]}
\newcommand{\Espi}[2]{\mathbb{E}_{#1}\left[ #2  \right]}
\newcommand{\Espci}[3]{\mathbb{E}_{#1}\left[ #2 \middle\vert #3 \right]}

\newcommand{\setN}{\mathbb{N}}
\newcommand{\setZ}{\mathbb{Z}}

\newcommand{\setR}{\mathbb{R}}

\newcommand{\Lie}[1]{\mathfrak{#1}}

\newcommand{\abs}[1]{\left| #1 \right|}
\newcommand{\normE}[1]{\left| #1 \right|_{E}}
\newcommand{\normEE}[1]{\left| #1 \right|_{E\otimes E}}
\newcommand{\distE}[2]{\mathbf{d}\left(#1,#2\right)}
\newcommand{\distEmulti}[3]{\mathbf{d}_{#1}\left((#2),(#3)\right)}

\newcommand{\ca}[1]{\mathcal{#1}}

\newcommand{\lambdatil}{\tilde{\lambda}}

\newcommand{\Xtil}{\tilde{X}}
\newcommand{\roughXtil}{\tilde{\mathbb{X}}}
\newcommand{\Rtil}{\tilde{R}}

\newcommand{\norm}[1]{ \left\vert\left\vert #1 \right\vert\right\vert }

\newcommand{\normA}[1]{\vert #1 \vert}

\newcommand{\intpart}[1]{\lfloor#1\rfloor}
\newcommand{\diln}{\sqrt{n^{-1}}}

\newtheorem{theorem}{Theorem}[section]

\newtheorem{definition}{Definition}[section]

\newtheorem{property}{Property}[section]
\newtheorem{lemma}{Lemma}[section]
\newtheorem{proposition}{Proposition}[section]
\newtheorem{corollary}{Corollary}[section]

\DeclareMathOperator{\Excursion}{Exc}
\DeclareMathOperator{\ExcursionE}{\widetilde{Exc}}
\DeclareMathOperator{\Embedding}{\iota}

\begin{document}
\title[Drifted L\'evy area from renormalization of Markov chains]{L\'evy area with a drift as a renormalization limit of Markov chains on periodic graphs}

\author{Olga Lopusanschi}
\email{olga.lopusanschi@upmc.fr}
\author{Damien Simon}
\email{damien.simon@upmc.fr}

\address{Laboratoire de Probabilit\'es et Mod\`eles Al\'eatoires (LPMA), UPMC-Sorbonne Universit\'e, UMR CNRS 7599, case 188, 4, pl. Jussieu, F-75252 Paris Cedex 5, France.}
\date{\today}

\begin{abstract}
A careful look at rough path topology applied to Brownian motion reveals new possible properties of the well-known L\'evy area, in particular the presence of an intrinsic drift of this area. Using renormalization limit of Markov chains on periodic graphs, we present a construction of such a non-trivial drift  and give an explicit formula for it. Several examples with explicit computations are included.
\end{abstract}

\maketitle

\tableofcontents

\section{Introduction}
 
    \subsection{General motivations}

Many papers deal with the convergence of discrete vector-valued processes to Brownian motion in the spirit of Donsker's theorem, i.e. using uniform convergence. This is useful for finite dimensional marginals but not for the study of differential equations: uniform topology is too weak to ensure proper approximation of integrals driven by a process of low regularity. Terry Lyons solved this problem by creating the \textit{rough path theory}, which Martin Hairer generalized to \textit{regularity structures}. In both cases, when working in dimensions higher than $2$, the processes are lifted to a more complex structure whose topology ensures the continuity of the solution map of SDEs (the \textit{Itô map}). 

Let us have a quick look at the theory of rough paths. The main idea is to build out of the initial processes involved more elaborate structures that allow to register all the relevant information. Loosely speaking, a (continuous) process in $\setR^d$ (with $d\geq2$) is considered as a first-level information, and we build the corresponding rough path by adding a few more levels. The number of necessary levels is determined by the regularity of the process (if $(X_t)_{t\geq 0}$ is of regularity $\alpha\in(0,1)$, we need $\intpart{1/\alpha}$ levels) and each level is given by an iterated integral of a tensorial product (a double integral for level two, a triple one for level three and so on). In particular, for a process $X=(X^1,\ldots,X^d)$, the second level is determined by the increments of the process (the first level) and the \textit{stochastic area}: a process on the space of $d\times d$ antisymmetric matrices given by 
\begin{equation}
   \label{eq:AreaCont}
A_{s,t}(X)=\left(\iint_{s<u<v<t}dX_u^idX_v^j - dX_u^jdX_v^i\right)_{1\leq i,j\leq d}
\end{equation}
with the convention $A_{0,t}(X)=A_{t}(X)$. A detailed introduction to rough paths can be found in \cite{artLejayIntroRough} and \cite{StFlour}, and more exhaustive treatments in \cite{FrizVicBook} or \cite{FrizHairerBook}. 

Up to now, rough path theory has been either applied to processes of lower regularity than the Brownian motion (for example, fractional Brownian motion) or to the Brownian motion itself, which resulted in a kind of \textit{rewriting of classical stochastic calculus}. Surprisingly -- and it is one of the main motivations for this paper -- it is actually possible to go beyond the stochastic calculus as we know it.
The second level of the Brownian rough path is made of a symmetric and an antisymmetric part. Whereas we have the choice between Itô and Stratonovich integration for the symmetric part, the antisymmetric one is given by the stochastic area of the Brownian motion, the \textit{Lévy area}, and is not affected by the choice of the integration scheme for the symmetric part. However, for some sequences approximating the Brownian motion, there is room on the second level for an extra term, the \textit{area drift} or \textit{area anomaly}. 

When approximating an SDE driven by a Brownian motion in a classical way, we are looking for a sequence which does not have an area anomaly at the limit. On the contrary, theorem \ref{thm:convergencegenerale} concentrates on a class of models which may exhibit a non-zero area anomaly at the limit. More precisely, we prove that Markov chains on periodic graphs (roughly speaking, graphs that are constructed by translation of a given finite graph) to which we add the area component, converge in the suitable rough path topology to $(B_t,\mathcal{A}_t+t\Gamma)_{t\in[0,\tau]}$, where $(B_t)_{t\ge 0}$ is the $d$-dimensional Brownian motion, $(\ca{A}_t)_t=(A_t(B))_t$ its associated Lévy area and $\Gamma$ is a $d\times d$ antisymmetric matrix which represents the \textit{area anomaly}.

One of our main goals is to show how getting a non-trivial area anomaly can be used to build new rough paths above Brownian motion and thus to go beyond classical calculus. Consequently, this is \emph{not} a question of classical stochastic integration but a kind of \emph{completion} of classical integration. Such deformations of classical stochastic calculus have been foreseen but never illustrated by explicit, simple, discrete Markovian processes with natural geometric interpretation. In \cite{artFrizGassiatLyons}, the authors exhibit the  area anomaly of a magnetic field but this only concerns the continuous case. The article \cite{artLejayLyonsArea} studies how the area anomaly influences the behaviour of SDEs. The present paper is an attempt to fill in some of the blanks mentioned above and to show that area anomaly from theorem \ref{thm:convergencegenerale} is in fact a generic property of some renormalized discrete Markov chain and should be taken into account in the study of SDEs originating from many discrete processes. 

\bigskip

\paragraph{\emph{Acknowledgements.}} D.~S. thanks Q.~Berger for valuable discussions on random walks in random environments at the origin of some constructions from the present paper. D.~S. and O.~L. want to thank the referee of the present article for valuable remarks and suggestions on bibliography on periodic graphs and their generalizations. D.~S. is partially funded by the Grant ANR-14CE25-0014 (ANR GRAAL).

 \subsection{Structure of the present article}
    \label{subsect:artStruct}
The present article is organized into five sections, the introduction being the first one. We present a very simple example of a Markov chain which exhibits a non-zero area anomaly in \ref{subsect:IntroExample}. 

After introducing some useful definitions like those of Markov chains $(X_n)_{n\in\setN}$ on periodic graphs and of stochastic signed area $(A_n(X))_{n\in\setN}$ associated to them, we state our main result in \ref{subsect:MainResult}. We then present a historical overview of some results that have a connection to ours. The section ends by a discussion on the consequences of theorem~\ref{thm:convergencegenerale} on the universality class of the multidimensional Brownian motion, and expresses some caveats about the continuous description of the large size limit of discrete models.

In section \ref{sect:genFrame}, we describe the general settings of our theorem and state some useful results: in particular, we present a decomposition of the process $(X_n,A_n(X))_{n\in\setN}$ which is based on excursion theory and inspired by renormalization ideas.

Section \ref{sect:Proof} is dedicated to the proof of our main result (theorem \ref{thm:convergencegenerale}) which is a generalization to our class of Markov chains of the Donsker-type theorem for rough paths from \cite{artBreuilFriz}. 

Finally, in section~\ref{sect:examples}, we present some applications of our result: we introduce a model in 3D for which we compute $\Gamma$ by numerical simulations and we give an example of application to an SDE. We end the section with a list of open questions which arise in connection with the area anomaly $\Gamma$.

   \subsection{An easy discrete example: rotating sums of Bernoulli r.v.}
   \label{subsect:IntroExample}
    
Let $(U_n)_{n\in\setN^*}$ be a sequence of independent Bernoulli random variables such that $\prob{U_1=1}=1-\prob{U_1=-1}=p$. We define two complex-valued processes $(Z_n)_{n\in\setN}$ and $(Z'_n)_{n\in\setN}$ in the following way: $(Z_n)_n$ is the random walk with increments chosen uniformly in $\{1,i,-1,-i\}$ and $(Z'_n)_n$ satisfies $Z'_0=0$ a.s. and, for $n\ge 1$, $Z'_n=\sum_{k=1}^n i^{k-1}U_k$. We set $X_n=\mathcal{R}(Z_n)$, respectively $X'_n=\mathcal{R}(Z'_n)$, and $Y_n=\mathcal{I}(Z_n)$, respectively $Y'_n=\mathcal{I}(Z'_n)$. A classical exercise in probability consists in checking that the laws of $Z_n/\sqrt{n}$ and of $Z'_n/\sqrt{2np(1-p)}$ both converge to a normal law $\ca{N}(0,1)$. Moreover, the processes $(X_n,Y_n)_n$ and $(X'_n,Y'_n)_n$ embedded in continuous time by linear interpolation converge both in law to a standard Brownian motion in the uniform topology. The discrete stochastic area of the process $(X_n,Y_n)_n$ is defined as 
\begin{equation}\label{eq:discreteareaintro}
A_n(X,Y)=\frac{1}{2} \sum_{1\leq k <l\leq n} (X_k-X_{k-1})(Y_l-Y_{l-1})-(Y_k-Y_{k-1})(X_l-X_{l-1})
\end{equation}
and $A(X',Y')_n$ is defined in the same way for the second process. The process $(A_n(X,Y)/n)_n$ embedded in continuous time is known to converge to the L\'evy area of the Brownian motion; the present paper deals with the rescaled L\'evy area $(A_n(X',Y')/(2np(1-p)))_n$ of the second process $(X'_n,Y'_n)_n$ and shows that, in the correct topology, it converges to the Lévy area of the Brownian motion with an \emph{additional} drift $\gamma$. This drift is easily evaluated as:
\begin{equation}\label{eq:rotatinggamma}
\gamma = \lim_{n\to\infty} \frac{\Esp{A(X',Y')_n}}{2np(1-p)} = \frac{(2p-1)^2}{8p(1-p)}
\end{equation}
and some additional computations show that the limit of the first higher cumulants of $(A_n(X',Y')/(2np(1-p))$ coincide with the ones of the classical L\'evy area. 

Figure \ref{fig:Z4Z} describes the process $(X'_n,Y'_n)$ as a Markov process in $\setZ^2$. In particular, the graph $G_0$ is induced on $\{0,1,2,3\}$ by the Markov chain $(X'_n,Y'_n)$ (we glue together the edges that connect a vertex to points of the same type). Figure \ref{fig:histogramZ4Z} presents histograms of the marginal laws of $(X_n,Y_n,A_n(X,Y))$ and $(X'_n,Y'_n,A(X',Y')_n)$ obtained by numerical simulations. This figure shows that, in the large $n$ limit and with the classical rescalings, the two processes are very similar, except for the additional drift $\gamma$ in the Lévy area (the \emph{area anomaly}). Up to our knowledge, such a limit process in continuous time has never been described.

Intuition about the similarities and differences between the two processes can be quickly explained by the following renormalization argument. The increments of $(Z_n-Z_{n-1})_n$ are independent, whereas only the increments $(Z'_{4n+4}-Z'_{4n})_n$ are independent. In a time interval $\{4n,4n+1,4n+2,4n+3,4n+4\}$, the increments of $(Z'_n)_n$ are bounded  and thus do not contribute to the Brownian limit in the uniform topology; however, during the same time interval, correlations among these increments produce non-centered random areas. From a renormalization point of view, the local time correlations are irrelevant for the uniform topology but relevant for the rough path topology. 

\begin{figure}
\begin{center}
\includegraphics{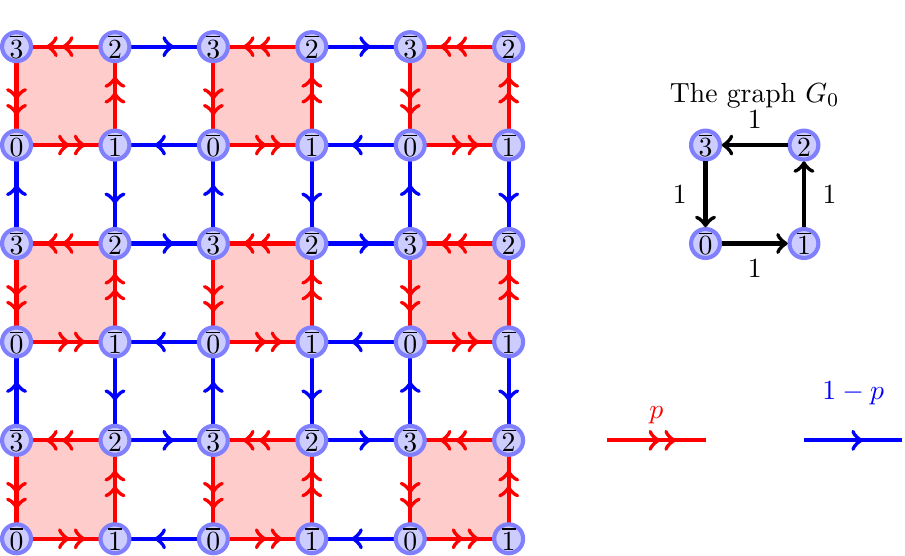}
\caption{Markov chain $(Z'_n)_n$ of section \ref{subsect:IntroExample}. On the left hand side, its authorized transitions on the periodic graph $G=\setZ^2$; on the top right hand side, its deterministic projection on the graph $G_0$ identified to $\setZ/4\setZ$. If $p\uparrow 1$, the chain tends to make more and more loops around the red shaded orbits.}
   \label{fig:Z4Z}
\end{center}
\end{figure}

\begin{figure}
\begin{center}
\includegraphics{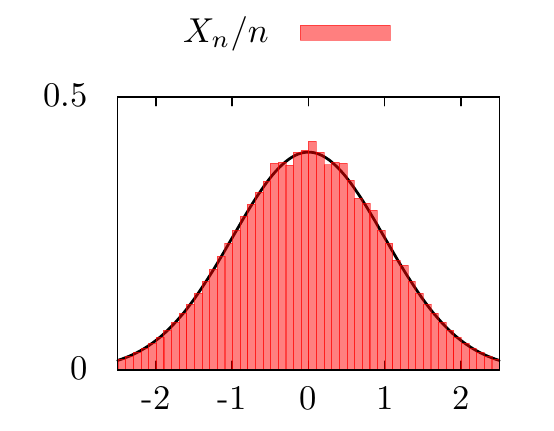}
\includegraphics{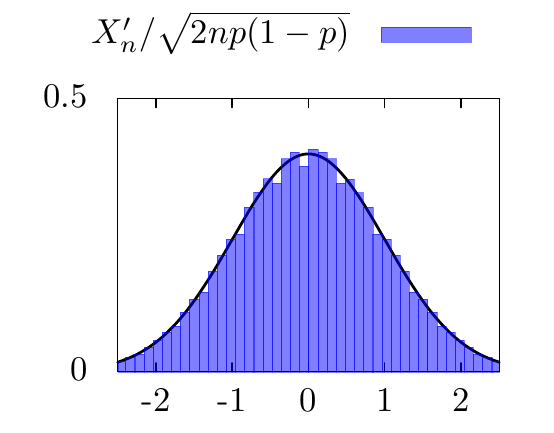}
\includegraphics{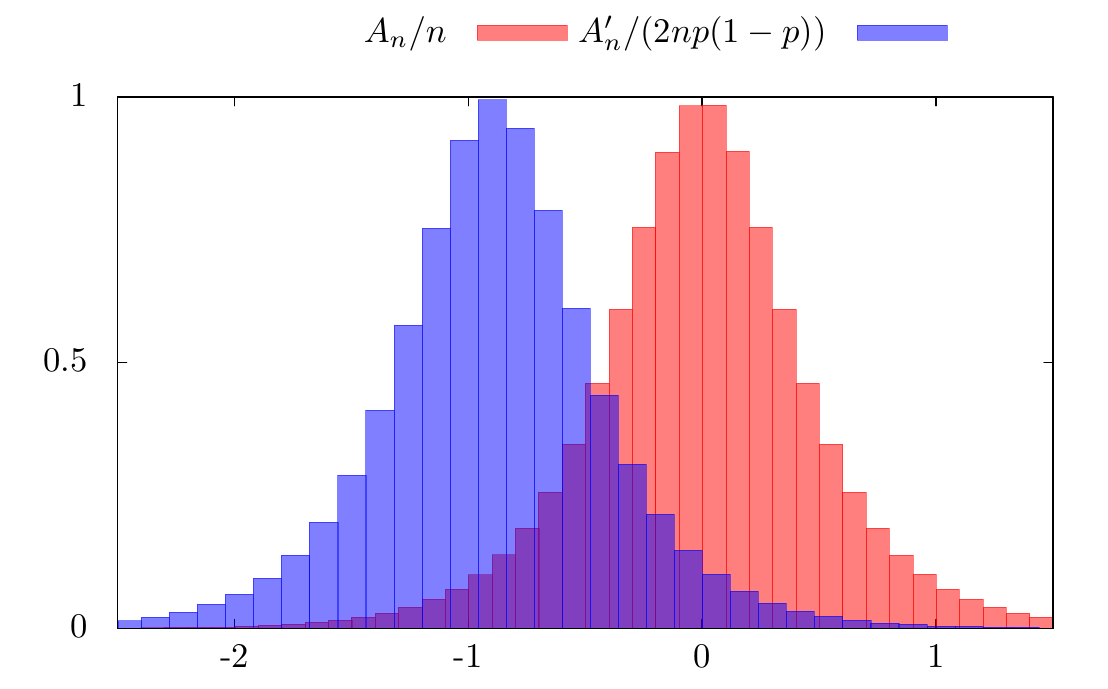}
\end{center}
\caption{\label{fig:histogramZ4Z}
Empirical distributions of $X_n/\sqrt{n}$ (top left), $X'_n/\sqrt{2np(1-p)}$ (top right), with the expected normal laws, and empirical distributions of the Lévy areas (below) $A_n/n$ (red, right) and $A_n(X',Y')/(2np(1-p))$ (blue, left). The parameters are chosen as $n=250000$ and $p=0.9$. Data are accumulated over 64000000 independent realizations. The empirical means of $A_n(X,Y)$ and $A'_n$ are $2.89\cdot 10^{-5}$ and $-0.88874$ and their empirical standard deviations are $0.500031$ and $0.499989$. The theoretical values are $\Esp{A_n(X,Y)/n}=0$, $\Esp{A_n(X',Y')/(2np(1-p))}=8/9=0.888\ldots$ and $\sigma=1/2$.}
\end{figure}

\subsection{Main result}
    \label{subsect:MainResult}
  \subsubsection{Main theorem}
     \label{subsubsect:MainTheorem}
  We start by properly defining the framework for our theorem.
\begin{definition}[periodic subgraph]
   \label{def:periodic subgraph}
Let $E$ be a finite-dimensional vector space. A periodic subgraph of $E$ is a infinite subset $G$ of $E$ such that:
\begin{enumerate}[(i)]
\item all the points are separated, 
\item $G$ is invariant under the translation action of a lattice $\Lambda\subset E$ on $G$.
\end{enumerate}
\end{definition}

\begin{property}
The graph can be decomposed as $G=\bigsqcup_{\lambda\in\Lambda}\lambda.G_0$ where $G_0$ is a finite subset of $G$ and $\lambda.G_0$ is the translation of $G_0$ by $\lambda\in\Lambda$.
\end{property}
This property means that any point $x$ of $G$ can be parametrized in a unique way as $(\lambda,x_0)$ where $\lambda\in\Lambda$ and $x_0\in G_0$. We write $\lambda=\pi_\Lambda(x)$ and $x_0=\pi_0(x)$ for the two projections. We use alternatively the notation $x$ or $(\lambda,x_0)$ for a point in $G$.

An illustration of this property is the decomposition of the model from example~\ref{subsect:IntroExample} detailed in figure~\ref{fig:Z4Z}.

\begin{definition}[invariant Markov chain on $G$]
Let $G$ be a periodic subgraph of $E$. A $G$-valued Markov chain $(X_n)_{n\in\setN}$ with transition law $Q$ on a probability space $(\Omega,\ca{F},\Prob)$ is $\Lambda$-invariant if and only if, for all $x,y,\in G$ and for all $\lambda\in\Lambda$, $Q(x+\lambda,y+\lambda)=Q(x,y)$.
\end{definition}

As it will be explained section~\ref{subsect:LambdaInvariantMarkov}, such a Markov chain $(X_n)_n$ induces a Markov chain $(\pi_0(X_n))_n$ on $G_0$. If the latter is irreducible, we can define a sequence of stopping times for it:
\begin{align*}
T_0	&= 0,
\\
T_{k+1}	&=	\inf \left\{ n>T_k: \pi_0(X_n)=\pi_0(X_0) \right\},		\qquad k\geq 0.
\end{align*}
Since $G_0$ is finite, for any initial law $\mu$ on $G_0$, $T_1$ has finite moments of all orders.

For an $E$-valued sequence $(x_n)_n$, we introduce its continuous rescaled version given by 
\begin{align*}
x^{(N)}_t=\frac{x_{\intpart{Nt}} + (Nt-\intpart{Nt}) (x_{\intpart{Nt}+1}-x_{\intpart{Nt}})}{\sqrt{N}}
\end{align*} 
as in the classical Donsker theorem. Then the rough path corresponding to $x^{(N)}$ is defined as
\begin{equation}
  \label{eq:Embedding}
\Embedding^{(N)}(x_\bullet,A_\bullet(x))_t := 
\Big(x^{(N)}_t, A_t(x^{N})) \Big)
\end{equation}
where $A_t(x^{N})$ is given by the formula \eqref{eq:AreaCont}. These variables belong to the space $G^2(E)$, which is described in section~\ref{subsect:G2E}. We denote by $\delta_{\epsilon}$ the standard homogeneous dilatation on $G^2(E)$: $\delta_{\epsilon}(x,a)=(\epsilon x,\epsilon^2 a)$.

We can now state our main theorem:

\begin{theorem}
    \label{thm:convergencegenerale}
Let $G$ be a $\Lambda$-periodic graph on a finite dimensional vector space $E$. Let $(X_n)_n$ be a $G$-valued $\Lambda$-invariant Markov chain on $G$ with bounded increments (i.e. there exists $R>0$ such that $\normE{X_{n+1}-X_{n}}\leq R$ a.s.) and such that $(\pi_0(X_n))_{n\in\setN}$ is irreducible.

Let $v=\Esp{T_1}^{-1}\Esp{X_{T_1}}\in E$ and $\beta=\Esp{T_1}\in\setR_+^*$. Let $(\Xtil_n)_{n\in\setN}$ be the $E$-valued process defined by $\Xtil_n=X_n - n v$. Up to a dimensional reduction and a linear transformation of the graph $G$, the covariance matrix of $\Esp{X_{T_1}}-T_1v$ may always be assumed to be $C I_n$ with $C>0$.

For any $\tau>0$, we have the following convergence in  distribution:
\begin{equation}
\left(\delta_{\sqrt{C^{-1}\beta}}\Embedding^{(N)}(\Xtil_{\bullet},A_{\bullet}(\Xtil))_t\right)_{0\leq t\leq \tau}
\underset{N\to\infty}{\longrightarrow}
\left(B_t,\ca{A}_t+t \Gamma\right)_{0\leq t\leq \tau}
\end{equation}
in the topology of $\ca{C}^{0,\alpha-\mathrm{Holder}}([0,\tau],G^2(E))$ for $\alpha<1/2$, where $B$ is a Brownian motion on $E$, $\ca{A}$ its Lévy area as defined by classical stochastic calculus and $\Gamma$ a constant antisymmetric matrix, the \emph{area anomaly}, given by \eqref{eq:Gamma}.
\end{theorem}

\subsubsection{Some historical background}
Until now, in the rough path setting, most of the convergence theorems have dealt with processes with i.i.d. centered increments. Based on the Stroock-Varadhan's result from \cite{artStrVar}, a Donsker theorem for rough paths in the $\ca{C}^{\alpha}([0,1],\setR^d)$ topology for $\alpha<1/2$ has been proved in theorem 3 from \cite{artBreuilFriz}, which our theorem 1.1 generalizes. In \cite{artCubature}, the discrete sequence converging to the Brownian rough path is constructed by the concatenation (in the rough path sense) of renormalized i.i.d. copies of the cubature formula on Wiener space. The main result of \cite{artRWLevy} gives sufficient conditions for convergence in distribution of a random walk on $G^N(\setR^d)$ to a Lévy process on $G^N(\setR^d)$ in a suitable rough path topology. While \cite{artCubature} and \cite{artRWLevy} are also generalizations of theorem 3 from \cite{artBreuilFriz}, they do not apply to the class of processes described in our theorem 1.1, as the increments of these processes are not necessarily i.i.d. Moreover, none of these results is concerned by the study of the area anomaly, as we will see that it is trivial when the discrete process has i.i.d. increments.

A more general setting, which encompasses that of rough paths, is given by random walks on different types of groups.  A result due to Wehn (see, for example, \cite{BreuRWLie}, theorem 1.3, for details) states that, when $\mu$ is a centered probability measure on a connected Lie group, $\mu^{*n}$ (the $n$th convolution of $\mu$) converges to the Wiener measure (under certain conditions on $\mu$). In \cite{artRWDiscreteGroups}, the main result states that, when $\mu$ is a probability measure with finite support on a discrete group of polynomial volume growth (nilpotent Lie groups, and in particular $G^N(\setR^d)$, are of polynomial volume growth), $\mu^{*n}$ converges to the heat kernel of a centered left-invariant sub-Laplacian on a certain simply connected nilpotent Lie group. In both cases, we deal with i.i.d. increments and no area anomaly is exhibited at the limit. However, in \cite{artRWDiscreteGroups} the possibility of a non-centered measure $\mu^{*n}$ is taken into account, just to show that the asymptotic behavior is similar to the non-centered case modulo a transformation by a multiplicative function (which is equivalent to re-centring $\mu$). What this shows in particular is that \emph{our area anomaly is not a question of the process drift} but a new phenomenon.

In the uniform topology, the convergence of processes similar to ours is widely studied. In \cite{artEnriqKif} and \cite{artKazUchi}, authors have already considered the convergence of random walks on periodic graphs and Markov chains on graphs respectively, and in \cite{artBaur} an invariance principle has been proved for a certain class of random walks in random environment. 
In \cite{artAlbanese} and \cite{artRWCristalLatt}, the authors study the convergence of a random walk, symmetric and non-symmetric respectively, on a lattice graph through the convergence of the corresponding discrete heat kernel.
Their lattice graph is a generalization to a Riemannian manifold of the notion of periodic graph from our article (a detailed theory on lattice graphs and the finite quotient graphs, as well as their properties, can be found in chapters 3 and 4 from \cite{bookCrystall}). While there is no room for an area anomaly for the reversible random walks from \cite{artAlbanese} (the reversibility of the process implies a zero area anomaly), the loops that can be present in the processes from \cite{artBaur} and \cite{artRWCristalLatt} might generate a non-zero limit stochastic area drift.

   \subsubsection{Discussion on the renormalization constant}
Let us now stop briefly to explain the choice of our renormalization constant, namely $v$. It is immediate to ask why we didn't simply set $\tilde{X}_n = X_n - \Esp{X_n}$. In this case, we argue that, first, we can not get an 
explicit infinite constant (of the type $nv$), second, it is a \textit{sufficient} but not a \textit{necessary} drift and finally, this would not allow us to get an explicit expression of $\Gamma$ of the type \eqref{eq:Gamma}.

As a consequence of the ergodic theorem applied to the Markov chain $(\pi_0(X_n))_n$, we could have also defined $v$ as
\begin{align*}
v = \sum_{x_0\in G_0}\nu(x_0)\sum_{y\in G}Q(x^*,y)(y-x^*)
\end{align*}
where $\nu$ the invariant probability of $(\pi_0(X_n))_n$ and $x^*$ a representative of the equivalence class of $x_0$ in $G$ ($z\sim z'$ if $\pi_0(z)=\pi_0(z')$). This is similar to the approach the authors adopt in \cite{artRWCristalLatt}.
Our choice of the expression of the renormalization constant is motivated by the desire to highlight the centering of the excursions, which is indispensable for applying the Donsker type theorem from \cite{artBreuilFriz}.

   \subsubsection{Symmetry of the process and area anomaly}    
In \cite{artRWCristalLatt}, the symmetry of the random walk on the quotient graph $X_0$ is sufficient but not necessary for $\rho_{\setR}(\gamma_p)$, the analogue of the constant $v$ called the asymptotic direction, to be zero (it is the condition $\gamma_p = 0$ that is equivalent to the symmetry of the random walk).

In our case, the symmetry of the random walk on $G_0$ is not necessary for $v$ to be zero either. In the example from~\ref{subsect:IntroExample}, the random walk on $G_0$ is not symmetric whatever value of $p$ we choose, and nevertheless $v = 0$ for $p=1/2$.

 On the other hand, the symmetry of $(X_n)_n$ (or its reversibility) will be a sufficient condition for $v=0$ and also for $\Gamma=0$. The most immediate example is provided by the framework of the classical Donsker theorem, i.e. the case when we are dealing with sums of i.i.d. centred r.v.'s.
 
However, we can have $(X_n)_n$ non-symmetric and simultaneously $v=0$ and $\Gamma \neq 0$ (see example \ref{subsect:IntroExample}), which shows that the area anomaly \textit{is not a product of the drift of the process}. We can also have $v\neq 0$ and $\Gamma=0$ in the case of i.i.d. non-centered r.v.'s (we need, of course, to re-center the variables before we pass to the limit).
 
   \subsubsection{Consequences of theorem \ref{thm:convergencegenerale}.}

The hypotheses of theorem \ref{thm:convergencegenerale} are satisfied in many models coming from statistical mechanics where jumps in space are often local. Up to our knowledge, the area anomaly is a new feature never described in any model, even if the examples that we present look natural. One may wonder whether this area anomaly is relevant. We now explain why it is the case.

The general philosophy beyond renormalization and large scale limits of discrete models is to build continuous models such that they are large scale limits of various discrete models and such that it is possible to compute directly with them. 

Phrased in a provocative way, our theorem implies in particular that \emph{a two-dimensional standard Brownian motion may not be the same as two independent one-dimensional Brownian motions} as soon as one wishes to use it to drive a stochastic differential equation. The difference lies in the area anomaly $\Gamma$ which is irrelevant at the level of the positions $(B_{t_k})_k$  but is relevant in non-linear SDEs. 

Thus, when several Brownian motions emerge in the description of the limit of discrete processes, the consequence of the previous theorem is that one \emph{needs in general to wonder} about the presence of area anomalies between components before writing down any stochastic integration. 

Hopefully in many cases, it is easy to prove without any calculation that the area anomaly is zero: as it has already been mentioned, this is the case for reversible processes. However, for irreversible Markov chains, especially useful in non-equilibrium statistical mechanics, one should expect in general a non-zero anomaly.

A detailed study of the area anomaly $\Gamma$ and its generalization to a larger class of processes will be present in a next paper in preparation.

\section{Tools and additional results}
   \label{sect:genFrame}
 
   \subsection{Additional results on $\Lambda$-invariant Markov chains}
     \label{subsect:LambdaInvariantMarkov}
\begin{property}
    \label{property:CMsurG0}
Let $(X_n)_{n\in\setN}$ be a $\Lambda$-invariant Markov chain on a periodic subgraph $G$ of $E$ as in definition~\ref{def:periodic subgraph}. Then the process $(\pi_0(X_n))_{n\in\setN}$ is a $G_0$-valued Markov chain.
\end{property}

\begin{proof}
Let $f$ be any bounded Borel function $G_0\to E$.
\begin{align*}
\Espc{ f(\pi_0(X_{n+1})) }{\ca{F}_n}
&=
\Espc{ \sum_{\lambda \in \Lambda} \indic{\pi_\Lambda(X_{n+1})=\lambda} f(\pi_0(X_{n+1})) }{\ca{F}_n} 
\\
&= \sum_{\lambda \in \Lambda}\Espc{  \indic{\pi_\Lambda(X_{n+1})=\lambda} f(\pi_0(X_{n+1})) }{\ca{F}_n} = \sum_{\lambda \in \Lambda} (Qg_\lambda)(X_n) 
\end{align*}
where $g_\lambda(x)=\indic{\pi_\Lambda(x)=\lambda} f(\pi_0(x))$ by the Markov property for $(X_n)_n$. The invariance of $Q$ gives now:
\begin{align*}
(Qg_\lambda)(x) &= \sum_{y\in G} Q(x,y) g_\lambda(y) 
= \sum_{y_0\in G} Q(x,y)\indic{\pi_\Lambda(y)=\lambda} f(\pi_0(y))
\\
&= \sum_{y_0\in G_0} Q((\pi_\Lambda(x),\pi_0(x)),(\lambda,y_0))f(y_0) = \sum_{y_0\in G_0} Q((0,\pi_0(x)),(\lambda-\pi_\Lambda(x),y_0))f(y_0)
\end{align*}
by $\Lambda$-invariance for $Q$. Summation over $\Lambda$ eliminates the dependence on $\pi_\Lambda(x)$ and we thus obtain:
\[
\Espc{ f(\pi_0(X_{n+1})) }{\ca{F}_n} = (Q_0 f)(\pi_0(X_n))
\]
with $Q_0(x_0,y_0)=\sum_{\lambda\in\Lambda} Q((0,x_0),(\lambda,y_0))$.
\end{proof}

Moreover, similar calculations show that the process $(\pi_\Lambda(X_n))_n$ knowing the process $(\pi_0(X_n))_n$ is a heterogeneous Markov chain whose rates\footnote{the probability of a jump between $\lambda$ and $\lambda'$ between times $k$ and $k+1$ is $Q_k(\lambda,\lambda' | \pi_0(X))= Q((\lambda,X_k),(\lambda',X_{k+1}))/Q_0(X_k,X_{k+1})$.} depend on the $(\pi_0(X_n))_n$.

\subsection{Decomposition into pseudo-excursions}
   \label{subsect:pseudoexcursions}
We start with a general definition of pseudo-excursions for an $E$-valued sequence:

\begin{definition}
  \label{def:pseudoExcOnE}
Let $(x_n)_{n\in\setN}$ be an $E$-valued sequence and $(T_k)_{k\in\setN}$ be a strictly increasing sequence in $\setN$ such that $T_0=0$ and $T_{k+1}-T_k=L_k>0$. We introduce the sequence $\lambdatil_p(x):=x_{T_{p}}$. Let $o$ be an additional cemetery point added to $E$. The pseudo-excursions $\ExcursionE^{(p)}(x)$ of the sequence $(x_n)_{n\in\setN}$ are then defined as $E\cup \{o\}$-valued processes through:
\begin{align*}
\ExcursionE^{(k)}(x)_n = 
\begin{cases}
x_{T_k+n}-\lambdatil_k(x) & \text{if $0\leq n\leq L_k$} \\
o 	& \text{if $n>L_k$}
\end{cases}
\end{align*}

\end{definition}

The global trajectory $(x_n)_n$ can be recovered from the excursions by:
\begin{equation}
 \label{eq:TrajectoryFromExc}
x_n = \lambdatil_{\kappa(n)}(x) + \ExcursionE^{(\kappa(n))}(x)_{n-T_{\kappa(n)}}
\end{equation}
where $\kappa(n)$ is the unique integer such that $T_{\kappa(n)}\leq n <T_{\kappa(n)+1}$.

We will now give a definition of pseudo-excursions which applies to a specific class of $G$-valued sequences we are interested in (with $G$ as in definition \ref{def:periodic subgraph}). For this purpose, we will slightly change definition~\ref{def:pseudoExcOnE}.

For $(x_n)_{n\in\setN}\in G^{\setN}$, we introduce the sequence of excursion times of $(\pi_0(x_n))_{n\in\setN}$ from its original point:
\begin{align*}
T_0	&= 0,
\\
T_{k+1}	&=	\inf \left\{ n>T_k: \pi_0(x_n)=\pi_0(x_0) \right\},		\qquad k\geq 0.
\end{align*}
 
\begin{definition}
  \label{def:pseudoExcOnG}
Let $(x_n)_{n\in\setN}$ be a $G$-valued sequence such that $(\pi_0(x_n))_{n\in\setN}$ is recurrent (i.e. each point of $G_0$ appears an infinity of times in the sequence). Set $\lambda_k(x) =\pi_\Lambda(x_{T_k})$ and $L_k=T_{k+1}-T_k$ ($L_k$ is the duration of an excursion). Let $o$ be an additional cemetery point added to $G$. The pseudo-excursions $\Excursion^{(k)}(x)$ of the sequence $(x_n)_{n\in\setN}$ are defined as $G\cup \{o\}$-valued processes through:
\begin{equation}
\label{eqdef:pseudoexc}
\Excursion^{(k)}(x)_n = 
\begin{cases}
x_{T_k+n}-\lambda_k(x) & \text{if $0\leq n\leq L_k$} \\
o 	& \text{if $n>L_k$}
\end{cases}
\end{equation}

\end{definition}
Although the above definition can be viewed as a particular case of definition~\ref{def:pseudoExcOnE}, its interest consists in exploiting the decomposition of elements of $G$ in $\Lambda\times G_0$-valued couples. This enables us to translate only the $\Lambda$-valued component and thus to start each pseudo-excursion from a point $y\in G$ such that $\pi_0(y)=\pi_0(x_0)$. Moreover, as we keep here close to the classical definition of excursions, we can deal with the Markov chain $(X_n)_{n\in\setN}$ (and not only $(\pi_0(X_n))_{n\in\setN}$) and make computations of the L\'evy area easier. In the rest of the article, we will prefer definition~\ref{def:pseudoExcOnG} when we talk of a (recurrent) $G$-valued sequence, and definition~\ref{def:pseudoExcOnE} will apply whenever we make a statement concerning any $E$-valued sequence.

One immediately checks that $\Excursion^{(k)}(x)_0=\pi_0(x_0)$ and $\Excursion^{(k)}(x)_{L_k}=\pi_0(x_0)+\lambda_{k+1}(x)-\lambda_k(x)$. Our construction of pseudo-excursions makes them invariant under translation by an element of $\Lambda$: $\Excursion^{(k)}(\mu+x)=\Excursion^{(k)}(x)$ where $\mu$ is any element of $\Lambda$. 

\begin{property} Let $(X_n)_n$ be a $\Lambda$-invariant Markov chain on the periodic graph $G$ such that the projection $(\pi_0(X_n))_n$ is irreducible. The $(G\cup\{o\})^\setN$-valued random variables $(\Excursion^{(k)}(X))_{k\in\setN}$ are independent and identically distributed.
\end{property}

\begin{proof}
The proof relies on the repetitive use of the strong Markov property. Let $n\in\setN^*$ and let $f_0,f_1,\ldots,f_n: (G\cup\{o\})^\setN\to \setR$ be bounded measurable functions. The random times $T_k$ are stopping times, which are finite almost surely. For $k\leq n-1$, the random variables $\ca{E}_k= \Excursion^{(k)}(X)$ are $\ca{F}_{T_{k}}$-measurable, and we obtain:
\begin{align*}
\Espci{x}{ f_0(\ca{E}_0)\ldots f_{n-1}(\ca{E}_{n-1}) f_n(\ca{E}_n) }{\ca{F}_{T_n}}
&= f_0(\ca{E}_0)\ldots f_{n-1}(\ca{E}_{n-1}) \Espci{x}{  f_n(\ca{E}_n) }{\ca{F}_{T_n}}
\end{align*}
The strong Markov property thus yields:
\begin{align*}
 \Espci{x}{  f_n(\ca{E}_n) }{\ca{F}_{T_n}} 
 &=
  \Espi{X_{T_n}}{f_n(\ca{E}_0)} = \Espi{(\lambda_n(x),\pi_0(x))}{f_n(\ca{E}_0)} \\
 & = 
  \Espi{(\lambda_0(x),\pi_0(x))}{f_n(\Excursion^{(0)}(X+\lambda_n-\lambda_0))}   
 = \Espi{(\lambda_0(x),\pi_0(x))}{f_n(\ca{E}_0)}
\end{align*}
a.s., where the last equality is deduced from the $\Lambda$-invariance of the process and of the pseudo-excursions. One now remarks that the last term does not depend anymore on $X_{T_{n}}$. By recursion, we obtain the final result:
\[
\Espi{x}{ f_0(\ca{E}_0)\ldots f_{n-1}(\ca{E}_{n-1}) f_n(\ca{E}_n) }
= \Espi{x}{  f_0(\ca{E}_0) }\Espi{x}{  f_1(\ca{E}_0) } \Espi{x}{  f_n(\ca{E}_0) }
\]
\end{proof}

\begin{corollary}
\label{cor:indepdeltalambda}
The random variables $(\lambda_{k+1}(X)-\lambda_{k}(X))_{k\in\setN}$ are also i.i.d.
\end{corollary}
\begin{proof}
This follows directly from $\lambda_{k+1}(X)-\lambda_{k}(X)= \Excursion^{(k)}(X)_{L_k}-\Excursion^{(k)}(X)_{0}$ and the independence of the pseudo-excursions.
\end{proof}

\textbf{Remark:} If $G$ is a periodic graph and $M\in GL_n(\setR)$, then $MG=\{ Mx; x\in G \}$ is again a periodic graph (with possibly degenerate vertices). If $(X_n)_n $ is a $\Lambda$-invariant Markov chain on $G$, then $(MX_n)_n$ is again a $M\Lambda$-invariant Markov chain on $MG$. We assume all through the paper that $\setR^n=\Span \Lambda$; if this is not the case, we embed the graph $G$ in the smaller space $\Span \Lambda$ isomorphic to some $\setR^n$.  Let $\mathtt{C}$ be the covariance matrix of the increment $\lambda_1(X)-\lambda_0(X)$, we always assume that $\mathtt{C}=M^*I_{n}M$. If $\rank \mathtt{C} < n$, we again embed our Markov chain in a smaller graph in a smaller space such that $\rank \mathtt{C}=n$. Then, up to reduction to a smaller space and up to an invertible linear transformation of the graph, we may always assume that $\mathtt{C}=I_n$. In particular, under $\mathtt{C}=I_n$, the Donsker embedding of the random walk $(\lambda_k(x))_k$ converges to a standard Brownian motion on $\setR^n$. Moreover, the covariance matrix $\mathtt{C}$ and the drift $v$ are the analogues of the Albanese metric and the asymptotic direction respectively from the article \cite{artRWCristalLatt}.

\textbf{Example of section \ref{subsect:IntroExample}.} The Markov chain $(Z'_n)$ fits into this framework with $G=\setZ^2$. The non-zero elements of the transition matrix $Q$ are represented in figure \ref{fig:Z4Z}. The matrix $Q$ is $\Lambda$-invariant with $\Lambda=(2\setZ)^2$. The set $G_0=\{(0,0),(1,0),(0,1),(1,1)\}$ may be identified to $\setZ/4\setZ$, and so the Markov process $(\pi_0(Z'_n))_n$ is actually deterministic and corresponds to the shift $x\mapsto x+1$ as in figure \ref{fig:Z4Z}.

\subsection{Area process and rough paths}

\begin{definition}[area sequence]
Let $E$ be a finite-dimensional vector space. Let $(e_i)_{1\leq i\leq d}$ be a basis of $E$. We write $x^{(i)}$ for the $i$-th coordinate of a vector $x\in E$ w.r.t. the basis $(e_i)_{1\leq i \leq d}$.
For any $E$-valued sequence $(x_n)_{n\in\setN}$, we introduce the sequence of antisymmetric $d\times d$ matrices $(A_n(x))_{n\in\setN}$ defined by $A_0(x)=A_1(x)=0$ and, for any $n\geq 2$,
\begin{equation}
  \label{formula:decompArea}
A^{ij}_n(x) = \sum_{1\leq k < l \leq n}  \left( (\Delta x^{(i)})_k (\Delta x^{(j)})_l - (\Delta x^{(j)})_k (\Delta x^{(i)})_l \right)
\end{equation}
with $(\Delta u)_k= u_k-u_{k-1}$ for any sequence $(u_n)_n$.
\end{definition}
This definition can be tied easily to the stochastic area $A_t(x^{(N)})$ of $x^{(N)}$ from formula~\ref{eq:Embedding}: it is easy to check that 
\begin{equation}
  \label{eq:AreaDiscCont}
A_t(x^{(N)})=\frac{ A_{\intpart{Nt}}(x) +  (Nt-\intpart{Nt})( A_{\intpart{Nt}+1}(x)
- A_{\intpart{Nt}}(x))}{N}
\end{equation}

\begin{property}[decomposition of an area sequence along excursions]
  \label{prop:areaDecomp}
Let $(x_n)_{n\in\setN}$ and $(T_k)_{k\in\setN}$ be as in definition~\ref{def:pseudoExcOnE}. Then the following decomposition holds:
\begin{equation}
A_{T_n}^{ij}(x) = \sum_{p=0}^{n-1} A^{ij}_{L_p}(\ExcursionE^{(p)}(x)) 
+ 
A^{ij}_{n} (\lambdatil(x))
\end{equation}

In the particular case when $(x_n)_{n\in\setN}$ and $(T_k)_{k\in\setN}$ are as in definition~\ref{def:pseudoExcOnG}, we have:
\begin{equation}
A_{T_n}^{ij}(x) = \sum_{p=0}^{n-1} A^{ij}_{L_p}(\Excursion^{(p)}(x)) 
+ 
A^{ij}_{n} (\lambda(x))
\end{equation}
\end{property}

\begin{proof}
By definition, the l.h.s. uses a double sum over $1\leq k<l \leq T_n$. We split the interval $\{1,2,\ldots, T_n\}$ into $J_p=\{T_p+1,\ldots,T_{p+1}\}$ for $p=0,\ldots,n-1$ and we classify the indices $k$ and $l$: either they are in the same subset $J_p$ or they belong respectively to $J_{p_1}$ and $J_{p_2}$ with $p_1<p_2$. 

In the first case, the sum over $T_r+1\leq k < l \leq T_{r+1}$ gives the area of the $r$-th excursion $A_{L_r}^{ij}(\Excursion^{(r)}(x))$.

In the second case, the sum over $k\in J_{p_1}$ and $l\in J_{p_2}$ factorizes into two sums, evaluated as telescopic sums respectively to $\lambda_{p_i+1}(x)-\lambda_{p_i}(x)$ for $i=1,2$. The remaining sum over $0\leq p_1<p_2 \leq n-1$ gives the (signed) area of $(\lambda_k(x))_{k\in\setN}$ between $0$ and $n$.
\end{proof}

We need a last lemma, easy to prove, from linear algebra, about the transformation of the area under a linear transformation $M$ of the space $E$.

\begin{lemma}[covariance of the area]
Let $(x_n)_n$ be an $E$-valued sequence and $M\in GL_n(E)$. The area process $(A_n^{ij}(x))_n$ of the sequence $(Mx_n)_{n}$ in $E$ is given by:
\begin{equation*}
A_n^{ij}(Mx) =  \sum_{1\leq k,l\leq n} M_{ik} M_{jl} A_n^{kl}(x)
\end{equation*}
\end{lemma}

\subsection{The group $G^2(E)$}
  \subsubsection{The general construction}
     \label{subsect:G2E}
  
In this section, we rewrite some results from the rough path theory from \cite{FrizVicBook} (in particular from chapter 7) in order for them to correspond to the case of a finite-dimensional vector space $E$ on $\setR$. We concentrate on the case that is of interest to this article, namely $N=2$. For more details and the general case $N\geq 2$ see \cite{FrizVicBook} or \cite{FrizHairerBook}.

We introduce the \textit{tensorial truncated algebra} $T^{(2)}(E)=\bigoplus_{k=0}^2E^{\otimes k}$, where $\otimes$ is the tensorial product on $E$ ($E^{\otimes 0}=\setR$) and $\bigoplus$ denotes a direct sum. Endowed with the multiplication law
\begin{align}
  \label{formula:operationT2E}
  (a_0,a_1,a_2)\otimes_2 (b_0,b_1,b_2) = (a_0b_0, a_0 b_1 + b_0 a_1, a_0 b_2 + b_0 a_2+a_1\otimes a_2),
\end{align}
it is a non-commutative algebra with unit element $(1,0_E,0_{E\otimes E})$.

To $x \in \ca{C}^{1-var}([s,t],E)$ (the set of all continuous paths of finite $1$-variation), we associate the element of $T^{(2)}(E)$ given by
\begin{align*}
S_{2}(x)_{s,t}=\left(1,\int_{s<u<t}dx_u,
        \int_{s<u_{1}<u_{2}<t}
        dx_{u_1}\otimes dx_{u_{2}}\right)\in T^{(2)}(E)
\end{align*}
This object satisfies \textit{Chen's relation}, i.e., for $0\leq s<r<t\leq 1$,
\begin{align}
  \label{formula:Chen}
S_{2}(x)_{s,t}=S_{2}(x)_{s,r}\otimes_2 S_{2}(x)_{r,t}
\end{align}
and, in this particular case, $\otimes_2$ can be viewed as a path concatenation operator. 

As in section 7.5.1 in \cite{FrizVicBook}, we define the set $G^2(E)$ by
\begin{align*}
G^2(E)=\{S_{2}(x)_{0,1}:\ x\in \ca{C}^{1-var}([0,1],E)\}
\end{align*}
We now denote by $\mathbb{X}=(1,\mathbb{X}^{(1)},\mathbb{X}^{(2)})$ an element of $G^2(E)$, where $\mathbb{X}^{(1)}$ stands for the first-order and $\mathbb{X}^{(2)}$ for the second-order increments. Implicitly, $\mathbb{X}_{s,t}=S_2(x)_{s,t}$ for some $x\in \ca{C}^{1-var}([0,1],E)$, and $\mathbb{X}_{t}=\mathbb{X}_{0,t}$.
Since the symmetrical part of $\mathbb{X}^{(2)}$ depends on $\mathbb{X}^{(1)}$ (as $\int ydy=\frac{1}{2}y^2$), we can cut off redundant information by transforming $\mathbb{X}^{(2)}$ into
$\mathbf{X}^{(2);i,j}_{t}=\int_{0}^{t}(X^{i}_{s}-X^{i}_{0})dX^{j}_s-\int_{0}^{t}(X^{j}_{s}-X^{j}_{0})dX^{i}_s$ for $1\leq i,j \leq dim(E)$. Under this new form, the element $\mathbb{X}$ belongs to the space $\bigoplus_{k=0}^2E^{\wedge k}$, where $\wedge$ is the antisymmetric tensor product on $E$: for $u,v\in E$, $u\wedge v = u\otimes v - v\otimes u$.

For commodity reasons, we can use a more informal notation by neglecting the first component (the identity element) of $\mathbb{X}$.

Of course, the elements of the type $(x_t,A_t(x))$, which are the ones we are interested in, belong to $G^2(E)$. In particular, we can isolate from the operation $\otimes_2$ a very important property of the stochastic area, namely, for $0\leq s<t$
\begin{align}
  \label{eq:PropArea}
 A_t^{ij}(x)=A_s^{ij}(x) + A_{s,t}^{ij}(x) + \dfrac{1}{2}(x_s^{(i)}(x_t-x_s)^{(j)} - x_s^{(j)}(x_t-x_s)^{(i)})
\end{align} 

\subsubsection{The Carnot-Caratheodory distance on $G^2(E)$}
  \label{subsect:CarnotCaraDist}
It is natural to ask what is the shortest path in $E$ for a given signature. The answer to this question allows to define the \textit{Carnot-Caratheodory norm on $G^2(E)$} by
\begin{align}
  \label{def:CarnotCaraNorm}
\norm{g}:=\inf\{\int_{0}^{1}\normA{dx}:\ x\in \ca{C}^{1-var}([0,1],E)\text{ and }S_{2}(x)_{0,1}=g\}
\end{align}
where $\normE{\cdot}$ is a restriction to $E$ of the Euclidean norm.

Since the norm thus defined is homogeneous ($\norm{\delta_\lambda g}=|\lambda|\norm{g}$ for $\lambda\in\setR$), symmetric ($\norm{g}=\norm{g^{-1}})$ and sub-additive ($\norm{g\otimes h}\leq\norm{g}+\norm{h}$), it induces a left-invariant, continuous metric $\mathbf{d}$ on $G^2(E)$ through the application 
\begin{align}
  \label{def:CarnotCaraDistance}
  \begin{array}{cccc}
  \mathbf{d}: & G^2(E) \times G^2(E) &\to      &\mathbb{R}_{+}\\
     & (g,h)                   &\mapsto    &\norm{g^{-1}\otimes_2 h}
  \end{array}
\end{align}
In this case, $(G^2(E),\mathbf{d})$ is a geodesic space (in the sense of definition 5.19 from \cite{FrizVicBook}). It is also a Polish space (corollary 7.50 from \cite{FrizVicBook}).

The Carnot-Caratheodory norm is difficult to use for practical estimations but we can give it a good upper bound:

\begin{proposition}
  \label{prop:majorDist}
There exists $\nu>0$ such that, for $\mathbf{d}$ defined as above, for any $\mathbb{X}\in G^2(E)$ and $0\leq s<t\leq 1$, we have:
\begin{align}
\distE{\mathbb{X}_{s}}{\mathbb{X}_{t}}=\norm{\mathbb{X}_{s,t}}
 \leq \nu\left(\normE{\mathbf{X}^{(1)}_{s,t}} + \normEE{\mathbf{X}^{(2)}_{s,t}}^{\frac{1}{2}}\right)
\end{align}
\end{proposition}

\section{Proof of theorem \ref{thm:convergencegenerale} and comments}
   \label{sect:Proof}
    \subsection{Proof of theorem \ref{thm:convergencegenerale} }
  We denote by $\abs{\cdot}$ the absolute value on $\setR$, by $\normE{\cdot}$ the euclidean norm on the finite-dimensional vector space $E$ and by $\normEE{\cdot}$ the induced matrix norm on $E\otimes E$: $\normEE{A}=\underset{\normE{x}=1}{\sup}\normE{Ax}$. We also use the norm $\norm{\cdot}$ on $G^2(E)$ and the associated distance $\distE{\cdot}{\cdot}$ (definition~\ref{def:CarnotCaraDistance}). We set, for $u,v\in G^2(E)^{l}$, $\distEmulti{l}{u^1,\ldots,u^l}{v^1,\ldots,v^l}=\sum_{i=1}^{l}\distE{u^i}{v^i}$: $\mathbf{d}_{l}$ is a distance on $G^2(E)^{l}$.
 
For any $n\in\setN$, we define $\kappa(n)$ as the unique integer such that $T_{\kappa(n)}\leq n<T_{\kappa(n)+1}$, where the $T_n$s are as in definitions~\ref{def:pseudoExcOnE} or~\ref{def:pseudoExcOnG} (as has already been done in section~\ref{subsect:pseudoexcursions}).

\begin{proof}
 Since $\lambdatil_k(\Xtil)=\lambda_k(X)-T_{k}v$ (with $\lambda_k(X)$ and $\lambdatil_k(\Xtil)$ as in definitions~\ref{def:pseudoExcOnG} and~\ref{def:pseudoExcOnE} respectively), the process $(\lambdatil_k(\Xtil))_k$ is an $E$-valued centered random walk (not $\Lambda$-valued because of the correction). As it has been stated in the theorem, up to a dimensional reduction and a linear transformation of the graph $G$, the covariance matrix of $\lambda_1(X)-T_1v$ may always be assumed to be $C I_n$ with $C>0$, so each ($E$-valued) increment has a covariance matrix equal to $C I_n$.

 The main idea of the proof of theorem \ref{thm:convergencegenerale} is to use the theory of pseudo-excursions from section~\ref{subsect:pseudoexcursions} and the decomposition from property \ref{prop:areaDecomp} in order to extract convergence to the standard Brownian rough path through the process $(\lambdatil_n(\Xtil))_{n\in\setN}$ using theorem 3 \cite{artBreuilFriz}, convergence to the area anomaly through the independence of pseudo-excursions, and tightness from additional results on pseudo-excursions. Consequently, the proof of theorem \ref{thm:convergencegenerale} is divided into $4$ steps: 
 
\begin{itemize}
\item lemma \ref{lem:step1}: convergence of the centered discrete process $\left(\Embedding^{(N)}(\lambdatil_\bullet(\Xtil),A_{\bullet}(\lambdatil(\Xtil)))_t\right)_{0\leq t\leq\tau}$
\item lemma \ref{lem:step2}: convergence of the extracted process $\left(\Embedding^{(N)}(\Xtil_{T_\bullet},A_{T_\bullet}(\Xtil))_t\right)_{0\leq t\leq \tau}$ and emergence of the area anomaly (drift) $\Gamma$
\item lemma \ref{lem:step3}: convergence of finite-dimensional marginals of the full process $\left(\Embedding^{(N)}(\Xtil_\bullet,A_{\bullet}(\Xtil))_t\right)_{0\leq t\leq\tau}$
\item lemma \ref{lem:step4}: tightness of the sequence $\left(\Embedding^{(N)}(\Xtil_\bullet, A_\bullet(\Xtil))\right)_{N\in\setN}$
\end{itemize}

\begin{lemma}\label{lem:step1}
 The process $\left(\Embedding^{(N)}\left(\lambdatil_\bullet(\Xtil),A_\bullet(\lambdatil(\Xtil))\right)_t\right)_{0\leq t\leq \tau}$ converges in distribution to the L\'evy lift on $G^2(E)$ of a Brownian motion $(B_t)_{t\ge 0}$:
\begin{align*}
\left(\delta_{\sqrt{C^{-1}}}\Embedding^{(N)}\left(\lambdatil_\bullet(\Xtil),A_\bullet(\lambdatil(\Xtil))\right)_t\right)_{0\leq t\leq \tau}
\xrightarrow[N\to\infty]{(d)} 
\left(B_t,\mathcal{A}_t\right)_{0\leq t\leq \tau}
\end{align*}
in the topology of $\ca{C}^{0,\alpha-\mathrm{Holder}}([0,\tau],G^2(E))$ for $\alpha<1/2$.
\end{lemma}

\begin{proof}
This is a direct consequence of the Donsker-type theorem for a sequence of i.i.d. centered $G^2(\setR^2)$-valued random variables from \cite{artBreuilFriz}. In this article, the authors use a central limit theorem for centered i.i.d. variables on a nilpotent Lie group in order to prove the convergence of finite-dimensional distributions, and Kolmogorov's criterion to prove the tightness of the sequence. 
\end{proof}


\begin{lemma}\label{lem:step2}
The sequence of processes $\left(\delta_{\sqrt{C^{-1}}}\Embedding^{(N)}(\Xtil_{T_\bullet},A_{T_\bullet}(\Xtil))_t\right)_{0\leq t\leq \tau}$ 
converges in distribution to
$\left(B_{t},\mathcal{A}_{t} + t\Gamma\right)_{0\leq t\leq \tau}$ 
in the topology of $\ca{C}^{0,\alpha-\mathrm{Holder}}([0,\tau],G^2(E))$ for $\alpha<1/2$, with $\Gamma$ given by~\eqref{eq:Gamma}.
\end{lemma}

This is the part of the proof where the area anomaly $\Gamma$ first appears. We will see that, between~\ref{lem:step1} and~\ref{lem:step2}, nothing changes on the first level of the new sequence, since the embedding is obtained by linear interpolation and therefore does not keep track of the trajectory between $T_n$ and $T_{n+1}$. Simultaneously, a complementary term appears on the second level, in the expression of the stochastic area. This is due to the fact that, whereas the specific trajectory of an excursion is not memorized, its area is registered in the continuous embedding.

\begin{proof}

We have trivially by definition~\ref{def:pseudoExcOnE}:
\begin{align*}
\Xtil_{T_n}= 
 \lambdatil_n(\Xtil)
\end{align*}
Moreover, property \ref{prop:areaDecomp} applied to $\Xtil$ gives:
\begin{align}
A^{ij}_{T_n}(\Xtil) & = 
       A^{ij}_{n} (\lambdatil(\Xtil))
       + 
      \sum_{p=0}^{n-1} A^{ij}_{L_p}(\ExcursionE^{(p)}(\Xtil))
\end{align}
Each term $A^{ij}_{L_p}(\ExcursionE^{(p)}(\Xtil))$ represents exactly the area of the $(p+1)$-th excursion and the total sum is the complementary second-level term mentioned above.

Let us decompose using \eqref{formula:decompArea}:
\begin{align*}
 A^{ij}_{L_p}(\ExcursionE^{(p)}(\Xtil))
   ={} &
   \sum_{1\leq k < l \leq L_p}  \left( (\Delta \ExcursionE^{(p)}(\Xtil)^{(i)})_k (\Delta \ExcursionE^{(p)}(\Xtil)^{(j)})_l \right.
  \\ & - 
  \left. (\Delta \ExcursionE^{(p)}(\Xtil)^{(j)})_k (\Delta \ExcursionE^{(p)}(\Xtil)^{(i)})_l \right) 
  \\ ={} &
 A^{ij}_{L_p}(\Excursion^{(p)}(X))
\\ & +
\left(\sum_{1\leq k < l \leq L_p}\left( (X_{T_{p}+l} - X_{T_{p}+l-1}) - (X_{T_{p}+k} - X_{T_{p}+k-1})\right)\right)^{(i)} v^{(j)}
\\& - 
v^{(i)} \left(\sum_{1\leq k < l \leq L_p}\left( (X_{T_{p}+l} - X_{T_{p}+l-1}) - (X_{T_{p}+k} - X_{T_{p}+k-1})\right)\right)^{(j)} 
\end{align*}
We set 
\begin{align*}
\textbf{Corr}^{ij}_p(X)  = &
\left(\sum_{1\leq k < l \leq L_p}\left( (X_{T_{p}+l} - X_{T_{p}+l-1}) - (X_{T_{p}+k} - X_{T_{p}+k-1})\right)\right)^{(i)} v^{(j)}
 \\ & 
 - 
v^{(i)} \left(\sum_{1\leq k < l \leq L_p}\left( (X_{T_{p}+l} - X_{T_{p}+l-1}) - (X_{T_{p}+k} - X_{T_{p}+k-1})\right)\right)^{(j)} 
\end{align*}
and we call this term the \textit{area drift correction}. Since the increments of $X$ are bounded by a certain $R>0$, we deduce that 
\begin{align*}
\left|\textbf{Corr}^{ij}_p(X)\right|\leq K_{v,R}L_p^{2}
\end{align*}
where $K_{v,R}$ is a constant depending on $v$ and $R$. Likewise, we obtain
\begin{align*}
\left|A^{ij}_{L_p}(\Excursion^{(p)}(X))\right|\leq K'_{R}L_p^{2}
\end{align*}
where $K'_R$ is a constant depending on $R$. We can thus conclude that all the $A^{ij}_{L_p}(\ExcursionE^{(p)}(\Xtil))$ are integrable. Moreover, these variables are i.i.d., since $A^{ij}_{L_p}(\Excursion^{(p)}(X))$ and $\textbf{Corr}^{ij}_p(X)$ depend only on $\Excursion^{(p)}(X)$.
Thus, by the law of large numbers the following convergence holds:
\begin{align*} 
\frac{1}{n}\sum_{p=0}^{n-1} A_{L_p}^{ij}(\ExcursionE^{(p)}(\Xtil)) \xrightarrow[n\to\infty]{\as}  \Esp{ A_{L_0}^{ij}(\ExcursionE^{(0)}(\Xtil))}
 =  \Esp{A_{L_0}^{ij}(\Excursion^{(0)}(X))}+\Esp{\textbf{Corr}^{ij}_0(X)}
\end{align*} 
Slutsky's theorem for metric spaces states that, for two sequences $(X_n)_n$ and $(Y_n)_n$ on a metric space $(S,\rho)$ and such that $\rho(X_n,X)\to 0$ and $\rho(X_n,Y_n)\to 0$, then $\rho(Y_n,X)\to 0$ (see, for example, \cite{bookBillConv}, theorem 3.1). Applying it to the sequences 
\begin{align*}
& X_N=\left(\Embedding^{(N)}\left(\lambdatil_\bullet(\Xtil),A_\bullet(\lambdatil(\Xtil))\right)_t\otimes(0,0,t\Gamma)\right)_{0\leq t\leq \tau}
\\ &
Y_N=\left(\Embedding^{(N)}(\Xtil_{T_\bullet},A_{T_\bullet}(\Xtil))_t\right)_{0\leq t\leq \tau}
\end{align*}
we can conclude by using the result from lemma~\ref{lem:step1} that
\begin{align*}
\left(\delta_{\sqrt{C^{-1}}}\Embedding^{(N)}(\Xtil_{T_\bullet},A_{T_\bullet}(\Xtil))_t\right)_{0\leq t\leq \tau}
\xrightarrow[N\to\infty]{(d)}
\left(B_{t},\mathcal{A}_{t} + t\Gamma\right)_{0\leq t\leq \tau} 
\end{align*}
where the coefficients of the $d\times d$ (with $d=dim(E)$) matrix $\Gamma$ are given by 
\begin{equation}
  \label{eq:Gamma}
  \Gamma^{ij}=C^{-1}(\Esp{A_{L_0}^{ij}(\Excursion^{(0)}(X))} + \Esp{\textbf{Corr}^{ij}_0(X)})
  \end{equation} 
The matrix $\Gamma$ is the announced area anomaly. It is immediate from definition~\ref{formula:decompArea} that  $A_{L_0}^{ij}(\ExcursionE^{(0)}(\Xtil)) = - A_{L_0}^{ji}(\ExcursionE^{(0)}(\Xtil))$, which implies that $\Gamma$ is antisymmetric.
\end{proof}

\begin{lemma}\label{lem:step3}
For any $t_1<t_2<\ldots<t_k\in \setR_{+}$, we have
\begin{align*}
&  \left(
   \delta_{\sqrt{C^{-1}\beta}}\Embedding^{(n)}\left(\Xtil_{\bullet},A_{\bullet}(\Xtil)\right)_{t_1}
   ,\ldots,
   \delta_{\sqrt{C^{-1}\beta}}\Embedding^{(n)}\left(\Xtil_{\bullet},A_{\bullet}(\Xtil)\right)_{t_k}
   \right)
\\ &
 \xrightarrow[n\to\infty]{(d)}
  \left((B_{t_1},\mathcal{A}_{t_1} + {t_1}\Gamma),\ldots,(B_{t_k},\mathcal{A}_{t_k} + {t_k}\Gamma)\right)
\end{align*}
\end{lemma}
In this lemma, we pass from the embeddings of an extracted sequence $\left(\Embedding^{(N)}(\Xtil_{T_\bullet},A_{T_\bullet}(\Xtil))_t\right)_{0\leq t\leq \tau}$ to the embeddings of the full sequence $\left(\Embedding^{(N)}(\Xtil_{\bullet},A_{\bullet}(\Xtil))_t\right)_{0\leq t\leq \tau}$. We show that in the term $\delta_{\diln}\left(\Xtil_{\intpart{nt}},A_{\intpart{nt}}(\Xtil)\right)$ the only part that counts at the limit is the one given by the excursions up to time $n$, i.e. $\delta_{\diln}\left(\Xtil_{\kappa(\intpart{nt})},A_{\kappa(\intpart{nt})}(\Xtil)\right)$. At the same time, the constant $\beta$ appears in the renormalization, since we have to take into consideration the approximate length of an excursion up to time $n$ $\frac{n}{\kappa(n)}$, and $\beta$ is precisely the a.s. limit of this sequence.

\begin{proof}
 
 Set $\roughXtil_{n}=\left(\Xtil_{n},A_{n}(\Xtil)\right)$. With the upper bound from proposition~\ref{prop:majorDist} (since it doesn't play an important role here, we suppose that $\nu=1$), for $t\in [0,\tau]$, we get the inequality:
\begin{align*}
   \distE{\delta_{\diln}\roughXtil_{T_{\kappa(\lfloor nt \rfloor)}}}{\delta_{\diln}\roughXtil_{\lfloor nt \rfloor}}
   \leq 
   \dfrac{1}{\sqrt{n}}
   \normE{\Xtil_{\lfloor nt\rfloor} - \Xtil_{T_{\kappa(\lfloor nt\rfloor)}}}
   +
    \dfrac{1}{\sqrt{n}}
    \normEE{A_{T_{\kappa(\lfloor nt\rfloor)},\lfloor nt\rfloor}(\Xtil)}^{\frac{1}{2}}
\end{align*}
We are going to use this decomposition in order to prove that 
$\distE{\delta_{\diln}\mathbb{X}_{T_{\kappa(\lfloor nt \rfloor)}}}{\delta_{\diln}\mathbb{X}_{\lfloor nt \rfloor}}$ converges in probability to $0$. 

We set $\Rtil=R+\normE{v}$. First, it is easy to see that if $k'\leq k<k''$, by triangular inequality we have a.s.
\begin{align}
   \label{formula:MajorX}
\normE{\Xtil_{k-k'}}\leq \sum_{l=1}^{k-k'}\normE{\Xtil_{l} - \Xtil_{l-1}}\leq \Rtil(k-k')\leq \Rtil(k''-k')
\end{align}
Next, since $A_{n}(\Xtil)$ is a $d\times d$ matrix (with $d=dim(E)$), we have
\begin{align}
  \label{formula:MajorA}
\normEE{A_{k-k'}(\Xtil)}
\leq 
\sum_{i=1}^{d}\max_{j=1,\ldots,d}\abs{A^{ij}_{k-k'}(\Xtil)}
\leq
d\Rtil^2(k-k')^{2}
\leq
d\Rtil^2 (k''-k')^{2}
\end{align}
Further on, by strong Markov property, for $\epsilon>0$, the Chebyshev's inequality, together with~\eqref{formula:MajorX},
implies, for the first term,
\begin{align*}
\prob{\dfrac{1}{\sqrt{n}}
   \normE{\Xtil_{\lfloor nt\rfloor} - \Xtil_{T_{\kappa(\lfloor nt\rfloor)}}}>\epsilon }
  &  \leq
    \dfrac{\Esp{\normE{\Xtil_{\lfloor nt\rfloor - T_{\kappa(\lfloor nt\rfloor)}}}^2}}{n\epsilon^2}
      \leq
       \dfrac{\Rtil^2 \Esp{(T_{\kappa(\lfloor nt\rfloor)+1} - T_{\kappa(\lfloor nt\rfloor)})^2}}{n\epsilon^2}
  \\ & = 
  \dfrac{\Rtil^2\Esp{T_1^2}}{n\epsilon^2}
\end{align*}
and, for the second term, together with~\eqref{formula:MajorA},
\begin{align*}
\prob{\dfrac{1}{\sqrt{n}}
   \normEE{A_{T_{\kappa(\lfloor nt\rfloor)},\lfloor nt\rfloor}(\Xtil)}^{\frac{1}{2}} >\epsilon}
     \leq
     \dfrac{\Esp{\normEE{A_{T_{\kappa(\lfloor nt\rfloor)},\lfloor nt\rfloor}(\Xtil)}}}{n\epsilon^2}
      \leq
     \dfrac{d\Rtil^2\Esp{T_1^2}}{n\epsilon^2}
\end{align*}
Hence, taking into consideration the fact that $\Esp{T_1^2}<\infty$, we obtain
\begin{align*}
  \label{formula:CVP}
  & \prob{ \distE{\delta_{\diln}\roughXtil_{T_{\kappa(\lfloor nt \rfloor)}}}{\delta_{\diln}\roughXtil_{\lfloor nt \rfloor}} > \epsilon}
    \leq
     \prob{\dfrac{1}{\sqrt{n}}
   \normE{\Xtil_{\lfloor nt\rfloor} - \Xtil_{T_{\kappa(\lfloor nt\rfloor)}}}> \dfrac{\epsilon}{2}}
  \\ & 
    +
    \prob{\dfrac{1}{\sqrt{n}}
    \normEE{A_{T_{\kappa(\lfloor nt\rfloor)},\lfloor nt\rfloor}(\Xtil)}^{\frac{1}{2}}
    > \dfrac{\epsilon}{2}}
       \underset{n\to\infty}{\longrightarrow} 0
\end{align*}
As $\kappa(n)$ is the number of excursions up to a time $n$ of $(\pi_0(X_n))_n$, the ergodic theory tells us that $\kappa(n)/n \to 1/\beta$ a.s. as $n\to\infty$ Consequently, the above convergence in probability combined with the result from lemma \ref{lem:step2} implies
  \begin{align*}
 \delta_{\sqrt{n^{-1}C^{-1}\beta}}\roughXtil_{\lfloor nt\rfloor}
 \xrightarrow[n\to \infty]{(d)} 
 (B_t,\mathcal{A}_t + t\Gamma)
 \end{align*}
What is now left to do is pass from $\roughXtil_{\intpart{nt}}$ to $\Embedding^{(n)}(\Xtil_{\bullet},A_{\bullet}(\Xtil))_{t}$, and in order to do that we have to study the convergence of $\normE{\Xtil_{\intpart{nt}+1}-\Xtil_{\intpart{nt}}}/\sqrt{n}$ and $\normEE{A_{\intpart{nt}+1}(\Xtil)-A_{\intpart{nt}}(\Xtil)}/n$.

We start with $\normE{\Xtil_{\intpart{nt}+1}-\Xtil_{\intpart{nt}}}/\sqrt{n}\leq \Rtil/\sqrt{n} \to 0$ a.s. as $n\to\infty$. Similarly, by formula~\ref{eq:PropArea} and using the fact that $\abs{ab-cd}/2\leq a^2+b^2+c^2+d^2$, we conclude to the following convergence in probability:
\begin{align*}
\dfrac{\normEE{A_{\intpart{nt}+1}(\Xtil)-A_{\intpart{nt}}(\Xtil)}}{n}
& \leq
\dfrac{\Rtil^2+\normE{\Xtil_{\intpart{nt}}-\Xtil_{T_{\kappa(\intpart{nt})}}}^2}{n}
\\ & \leq 
\dfrac{\Rtil^2((T_{\kappa(\intpart{nt})+1} - T_{\kappa(\intpart{nt})})^2+1)}{n}\overset{\ca{P}}{\underset{n\to\infty}{\longrightarrow}} 0 
\end{align*}
We conclude by Slutsky's theorem that
\begin{align*}
\delta_{\sqrt{C^{-1}\beta}}\Embedding^{(n)}(\Xtil_{\bullet},A_{\bullet}(\Xtil))_{t}  
\xrightarrow[n\to\infty]{(d)}
\left(B_t,\ca{A}_t+ t\Gamma \right)
\end{align*}
It is now easy to pass to the multivariate case. Choose $t_1<t_2<..<t_l \in \setR_{+}$. Then we have immediately
 \begin{align*}
  & \prob{\distEmulti{l}{\delta_{\diln}\roughXtil_{T_{\kappa(\intpart{nt_1}}},\ldots,\delta_{\diln}\roughXtil_{T_{\kappa(\intpart{nt_l})}}}{\delta_{\diln}\roughXtil_{\intpart{nt_1}},\ldots,\delta_{\diln}\roughXtil_{\intpart{nt_l}}} > \epsilon}
  \\&
  \leq
  \sum_{i=1}^l \prob{\distE{\delta_{\diln}\roughXtil_{\intpart{nt_i}}}{\delta_{\diln}\roughXtil_{T_{\kappa(\intpart{nt_i})}}} >\dfrac{\epsilon}{l}}  \underset{n\to\infty}{\longrightarrow} 0
 \end{align*}
 Applying once again the result from lemma \ref{lem:step2}, we obtain
 \begin{align*}
  \left( \delta_{\sqrt{n^{-1}C^{-1}\beta}}\roughXtil_{t_1},\ldots,\delta_{\sqrt{n^{-1}C^{-1}\beta}}\roughXtil_{t_k} \right)
 \xrightarrow[n\to\infty]{(d)}
  \left( (B_{t_1},\mathcal{A}_{t_1} + {t_1}\Gamma),\ldots,(B_{t_k},\mathcal{A}_{t_k} + {t_k}\Gamma) \right)
 \end{align*}
 
and we conclude by applying Slutsky's theorem as in the univariate case.
 \end{proof}


\begin{lemma}
  \label{Kolmo}\label{lem:step4}
The sequence $\left(\Embedding^{(n)}(\Xtil_{\bullet},A_{\bullet}(\Xtil))\right)_{n\ge 0}$ is tight in $\alpha$-H\"{o}lder topology for $\alpha<1/2$.
\end{lemma}

\begin{proof}
As in \cite{artBreuilFriz}, we apply here the Kolmogorov's criterion. In order to do so, it will be enough to prove that, for $\tau>0$ fixed, for any $p>1$ there exists a positive constant $c_p$ such that, for all $0\leq s<t\leq \tau$,
\begin{align*}
\sup_n \Esp{\distE{\Embedding^{(n)}(\Xtil_{\bullet},A_{\bullet}(\Xtil))_t}{\Embedding^{(n)}(\Xtil_{\bullet},A_{\bullet}(\Xtil))_s}^{4p}} \leq c_p |t-s|^{2p-1}
\end{align*}
since $(2p-1)/(4p)\to 1/2^{-}$ as $p\to\infty$.

Choose $a>0$. By proposition~\ref{prop:majorDist} and applying to $(X_{n})_{n\in\mathbb{N}}$ the Markov property, we get:
\begin{align*}
& \Esp{\distE{\Embedding^{(n)}(\Xtil_{\bullet},A_{\bullet}(\Xtil))_t}{\Embedding^{(n)}(\Xtil_{\bullet},A_{\bullet}(\Xtil))_s}^{a}}
=
\Esp{\norm{\Embedding^{(n)}(\Xtil_{\bullet},A_{\bullet}(\Xtil))_{s,t}}^{a}}
\\ & =
\Esp{\Espc{\norm{\Embedding^{(n)}(\Xtil_{\bullet},A_{\bullet}(\Xtil))_{s,t}}^{a}}{\Xtil_{\intpart{ns}}}}
=
\Esp{\Espi{\Xtil_{\intpart{ns}}}{\norm{\Embedding^{(n)}(\Xtil_{\bullet},A_{\bullet}(\Xtil))_{t-s}}^{a}}}
\\ & =
\Esp{\norm{\Embedding^{(n)}(\Xtil_{\bullet},A_{\bullet}(\Xtil))_{t-s}}^{a}}
\end{align*}
Since $\left(\Embedding^{(n)}(\Xtil_{\bullet},A_{\bullet}(\Xtil))_t\right)_{0\leq t\leq \tau}$ is constructed by linear connections between the points $\Embedding^{(n)}(\Xtil_{\bullet},A_{\bullet}(\Xtil))_{k/n}$ for $k=0,..,\intpart{n\tau}$, the properties of geodesic interpolation imply that it is sufficient to prove that
\begin{align*}
\dfrac{1}{n^{2p}}\Esp{\norm{\roughXtil_k}^{4p}}\leq c_p\left(\dfrac{k}{n}\right)^{2p}
\end{align*}
for $k=0,..,\intpart{n\tau}$, uniformly over $n\ge 1$. As in \cite{artBreuilFriz}, this follows immediately if we prove, for all $p>1$,
\begin{align*}
\Esp{\norm{\roughXtil_{n}}^{4p}} = O(n^{2p})
\end{align*}
Here, Chen's relation (formula~\eqref{formula:Chen}) gives
\begin{align*}
\roughXtil_{n}=\roughXtil_{T_{\kappa(n)}}\otimes_2 \roughXtil_{T_{\kappa(n)},n}
\end{align*}
where $\otimes_2$ is the product on $G^2(E)$ from section~\ref{subsect:G2E} (it can also be interpreted as a path concatenation operator). As mentioned in section~\ref{subsect:CarnotCaraDist}, the norm $\norm{\cdot}$ is sub-additive. Using strong Markov property and the inequality:
\begin{equation}
\forall a,b\ge 0,\qquad (a+b)^{p}\leq 2^p(a^{p}+b^{p}), \label{eq:star}
\end{equation} we arrive to an initial upper bound:
\begin{align*}
\Esp{\norm{\roughXtil_{n}}^{4p}}
\leq
2^{4p}\left(\Esp{\norm{\roughXtil_{T_{\kappa(n)}}}^{4p}}+\Esp{\norm{\roughXtil_{T_{\kappa(n)},n}}^{4p}}\right)
\end{align*}
On one hand, as $\kappa(n)\leq n$ a.s., we have
\begin{align*}
\Esp{\norm{\roughXtil_{T_{\kappa(n)}}}^{4p}}
\leq
\max_{l=1,\ldots,n}\Esp{\norm{\roughXtil_{T_l}}^{4p}}
=O(n^{2p})
\end{align*}
since $\roughXtil_{T_l}$ is a product of $l$ centred i.i.d. variables ($\roughXtil_{T_l}=\roughXtil_{T_1}\otimes_2\roughXtil_{T_1,T_2}\otimes_2\ldots\otimes_2\roughXtil_{T_{l-1},T_l}$), and therefore $\Esp{\norm{\roughXtil_{T_l}}^{4p}}=O(l^{2p})$ as it was proved in~\cite{artBreuilFriz}.

On the other hand, proposition~\ref{prop:majorDist} (with the convention $\nu=1$), and the inequality \eqref{eq:star}, together with the upper bounds from~\eqref{formula:MajorX} and~\eqref{formula:MajorA},
give
\begin{align*}
\Esp{\norm{\roughXtil_{T_{\kappa(n)},n}}^{4p}}
\leq
2^{4p}\Rtil^{4p}(d^{2p}+1)\Esp{(T_{\kappa(n)+1}-T_{\kappa(n)})^{4p}}
=
2^{4p}\Rtil^{4p}(d^{2p}+1)\Esp{T_1^{4p}}
\end{align*}
We therefore obtain
\begin{align*}
\Esp{\norm{\roughXtil_{n}}^{4p}}
\leq
2^{4p}\left(O(n^{2p}) + 2^{4p}\Rtil^{4p}(d^{2p}+1)\Esp{T_1^{4p}}\right)
= O(n^{2p})
\end{align*}
which achieves the proof.
\end{proof}
The results of lemmas \ref{lem:step3} and \ref{Kolmo} (convergence of finite-dimensional marginals plus tightness) put together give us the final result.
\end{proof}
 
   \subsection{Properties of the area anomaly}
Let us briefly discuss the formula of $\Gamma$ given by \eqref{eq:Gamma}. The main term, the one that we concentrate on, is given by $\Esp{A_{L_0}^{ij}(\Excursion^{(0)}(X))}$, which is the expectation of the stochastic area of an excursion. This is what we were intuitively expecting: the oscillations of the process along an excursion are not visible at the limit in the uniform topology but they generate stochastic area that influences the second level of a rough path through a drift.

The second term, $\Esp{\textbf{Corr}^{ij}_0(X)}$, comes from the fact that the excursions are not necessarily centered. It can be seen as a trace on the second level of the rough path of the fact that the excursions have been re-centered. Of course, if the excursions have zero mean from the beginning as in example \ref{subsect:IntroExample}, this correction term is zero.

These remarks imply that the area anomaly of a Markov chain on periodic graphs \textit{depends on the drift and on the stochastic area of a pseudo-excursion}. Let us see how this is different from the area drift generated exclusively by the drift of the process, as in \cite{artCubature}. Notice that a deterministic drift appears at level $2$ in corollary 3.4 but it is not an area anomaly as it depends entirely on the drift we assign to the Brownian motion $(B_t)_{t\geq 0}$ (and consequently to $W$). However, in this case, we recover supplementary terms on the second level, and we do not have the area+drift scheme anymore. It is important that this corollary allows us to consider the convergence of drifted processes and their stochastic areas without centering them, which is indispensable in the Donsker-type theorems. 

The presence of area anomaly in the limit stochastic area is the reason why the Itô map sometimes fails to be continuous in the uniform topology, and thus, in order to correctly approach an SDE driven by the Brownian motion, we need to be sure to get a zero area anomaly, as it has already been mentioned in the introduction. A discussion on this topic, as well as a method of approaching the Brownian motion and its Lévy area through the rough path of a cubature formula on Wiener space can be found in  \cite{artCubature}.

As it has already been mentioned, the area anomaly is zero when the process is symmetric/reversible, in particular when we consider the sum of centered i.i.d. r.v.'s as in the classical Donsker setting.
We can use this fact to construct a sequence of processes which generate area anomaly "artificially": we start by a process with piecewise linear i.i.d. centered increments and we replace every increment by a path of bounded variation and a stochastic area with a non-zero mean. For example, in the case of the sequence constructed by concatenating i.i.d. copies of the cubature formula on Wiener space from \cite{artCubature}, we concatenate to every copy of the cubature a centered random variable $C_n$ from $\ca{C}^{1-var}([0,1],\setR^d)$ such that the $C_n$'s are i.i.d. and their stochastic areas are of non-zero mean $\Gamma$. In this case, by the law of large numbers, the area anomaly is equal to $\Gamma$ (modulo some renormalization constant).

\section{Applications and open questions}
\label{sect:examples}

\subsection{Application to an SDE} 
   \label{subsect:SDE}
Stochastic differential equations may arise as limits of discrete difference equations. We consider here the simple case of a two-dimensional process $(X_n,Y_n)_{n\in\setN}$ and the difference equation:
\begin{equation}
\label{eq:sdeexample}
U_{n+1}-U_n = \epsilon [ f(U_n) (X_{n+1}-X_n) + g(U_n) (Y_{n+1}-Y_n)]
\end{equation}
where $(U_n)$ is an $\setR$-valued process. If $(X_n,Y_n)_n$ is a random walk converging towards a standard Brownian motion in $\setR^2$ 
under Donsker's embedding and if $\epsilon$ varies as $\epsilon=1/\sqrt{N}$ where $N$ is the parameter of the Donsker embedding, then $U_{n}$ converges to the solution of the SDE
\begin{equation}
dU_t = f(U_t)  dB_t^{1} + g(U_t) dB_t^{2}
\end{equation}
with the Itô prescription. 

We may now substitute a Markov chain on a periodic graph, like the ones described previously, to the random walk. Now, the suitable framework is rough path theory with an area anomaly. If we choose the process from example \ref{subsect:IntroExample}, an easy computation inspired from the proof of theorem \ref{thm:convergencegenerale}, which consists in dividing the process into excursions of length 4, shows that the limit process solves the SDE:
\begin{subequations}
\begin{align}
dU_t =& f(U_t) dB_t^{1} + g(U_t) dB_t^{2} \\
&+  \frac{1}{2} [f'(U_t)f(U_t) +g'(U_t) g(U_t)] K dt \label{eq:classicalterm}
\\
&+ \frac{1}{2} [f'(U_t)g(U_t)-f(U_t)g'(U_t) ] \gamma dt 
\label{eq:anomalousterm}
\end{align}
\end{subequations}
where $\gamma$ is the area anomaly \eqref{eq:rotatinggamma} and $K$ is the variance of the variables $\lambda_k = X_{4k}-X_{4(k-1)}$. The term \eqref{eq:classicalterm} is a well-known term in classical stochastic calculus similar to the Itô/Stratonovitch correction. The term \eqref{eq:anomalousterm} is new and requires the area anomaly. 

As explained in \cite{FrizVicBook}, this behavior is indeed general and generalizes to any Markov chain on a periodic graph satisfying theorem \ref{thm:convergencegenerale}, \emph{mutatis mutandis}. Here again, one notices that both terms \eqref{eq:classicalterm} and \eqref{eq:anomalousterm} are produced by the coarse-graining procedure based on excursions which leads to non-trivial renormalization terms.

\subsection{A three dimensional model with a non-trivial area anomaly.}
  \label{subsect:3D}
We extend the model presented in the introduction to dimension three. This extension is interesting for two main reasons: no particular role is played by the roots of unity as in the introduction and we may choose arbitrary jump rates; moreover, the area anomaly is now an antisymmetric three-by-three matrix which can be arbitrary. Such a process can then be used to obtain a Brownian motion on $SU(2)$ with an area anomaly with the classical identification between $\Lie{su}(2)$ and $\setR^3$ and solving the rough differential equation $dU_t= U_t dB_t$.

The graph $G$ is $\setZ^3$, the lattice $\Lambda$ is $(2\setZ)^3$ and the fundamental domain is thus $G_0=(\setZ/2\setZ)^3$. The only jumps allowed are those between $x$ and $x\pm e_k$ where $e_k$ is one of the three vectors of the canonical basis. The coefficients $Q(x,x')$ depend only on the classes modulo $2$ of each coordinate of $x$ and $x'$. A jump $\pm e_k$ changes the modulo class by $1$ on the coordinate $k$. Once projected onto $G_0$, the two jumps $x\pm e_k$ give the same transition on the cube. $Q$ is then parametrized by $8\times 6$ parameters (the cardinal of $Q_0$ times the number of directions).

In order to kill in a natural way the asymptotic drift $v$ of the process, we assume a central symmetry such that $Q(x,x\pm e_k)= Q(x+(1,1,1), x+(1,1,1)\mp e_k)$. The model is thus parametrized by $24=8\cdot 6/2$ parameters. In the generic case, the area anomaly $\Gamma$ is non-zero.

\begin{figure}
\begin{center}
\begin{tabular}{| l | l | c | c | c | c | c | c |}
\hline
Origin in $G$ & Proj. on $G_0$ &  $+e_1$ & $-e_1$ & $+e_2$ & $-e_2$ & $+e_3$ & $-e_3$
\\
\hline 
$(2k,2l,2m)$ & $(0,0,0)$ &  $u/2$ & $v/2$ & $u/2$ & $v/2$ & 0 & 0 	
\\ 
\hline
$(2k+1,2l,2m)$ & $(1,0,0)$ & 0 & 0 & $u/2$ & $v/2$ & $u/2$ & $v/2$ 	
\\
\hline
$(2k,2l+1,2m)$ & $(0,1,0)$ & $u/3$ & $v/3$ & $v/3$ & $u/3$ & $u/3$ & $v/3$ 
\\
\hline
$(2k+1,2l+1,2m)$ & $(1,1,0)$ & $v/2$ & $u/2$ & 0 & 0 & $u/2$ & $v/2$  
\\
\hline
$(2k,2l,2m+1)$ & $(0,0,1)$ & $u/2$ & $v/2$ & 0 & 0 & $v/2$ & $u/2$ 	
\\
\hline
$(2k+1,2l,2m+1)$ & $(1,0,1)$ & $v/3$ & $v/3$ & $u/3$ & $v/3$ & $v/3$ & $u/3$	
\\
\hline
$(2k,2l+1,2m+1)$ & $(0,1,1)$ & 0 & 0 & $v/2$ & $u/2$ & $v/2$ & $u/2$
\\
\hline
$(2k+1,2l+1,2m+1)$ & $(1,1,1)$ &  $v/2$ & $u/2$ & $v/2$ & $u/2$ & 0 & 0 	
\\
\hline
\end{tabular}
\end{center}
\caption{\label{fig:cubedynamics}Parameter of the dynamics of the cubic model used for the numerical results. The numerical simulations are made for $u=9/10$ and $v=1-u=1/10$ in order to bias the Markov chain to stay in a cube, so that it can develop a non-zero area anomaly $\Gamma$.}
\end{figure} 

Simulations are made for parameters chosen as in figure \ref{fig:cubedynamics} with $u=9/10$ and $v=1/10$ for over $64.10^6$ simulations and the process is observed at time $n=40000$. We obtain the following values for both coordinates of $X_n^{(i)}/\sqrt{n}$:
\begin{itemize}
\item the empirical means  are $-0.0025$, $-0.0020$ and $-0.0025$.
\item the empirical covariance matrix has three coefficients $0.03001$ on the diagonal and the other coefficients are all below $10^{-8}$.
\item the empirical third cumulants are all three below $10^{-6}$
\item the kurtosis are all three $3.0007$, $3.0006$ and $3.0004$. 
\end{itemize}  
The empirical values for the area $A^{ij}_n/(\sigma_i\sigma_j n)$ normalized by the empirical standard deviations of the coordinates are:
\begin{itemize}
\item empirical mean $\Gamma^{12}=1.500$, $\Gamma^{23}=1.500$ and $\Gamma^{31}=-1.500$ (the area anomalies),
\item empirical standard deviations are $0.5011$, $0.5011$ and $0.5011$.
\item empirical third cumulants $-1.46\cdot 10^{-4}$, $-6.9\cdot 10^{-5}$ and $1.1\cdot 10^{-4}$
\item empirical fourth cumulants $0.12533$, $0.12535$ and $0.12532$.
\end{itemize}
All the cumulants correspond to a normal law for the coordinates and a Lévy drifted area for the area process.

\subsection{Application to a stochastic differential equation on $SU(2)$}
   \label{subsect:SU2}
The previous example can be used to study the effect of the area anomaly on an SDE. We identify $\setR^3$ with $\Lie{su}(2)$ through the correspondence $(x,y,z)\to x \sigma_1 +y \sigma_2 + z \sigma_3$ where the $\sigma_1$, $\sigma_2$, $\sigma_3$ are the three Pauli matrices. We may now build the process with values in $SU(2)$ defined by:
\begin{equation}
U_{n+1}^{\epsilon} = U_n^{\epsilon} \frac{I + \epsilon i (X_{n+1}-X_n)}{I - \epsilon i (X_{n+1}-X_n)}
\end{equation}
If $(X_{n})_n$ is a standard random walk on $\Lie{su}(2)$, the process $(X_{n})_n$ converges in law after rescaling to the Brownian motion on $\Lie{su}(2)$ by Donsker's theorem and $U_n^{\epsilon}$ with the correct scaling of $\epsilon$ converges in law to the canonical Brownian motion on $SU(2)$. 

If we replace the random walk by the process $(X_n)_n$ from the previous section, $X_{n}$ converges in law after rescaling to the Brownian motion on $\Lie{su}(2)$ with an area anomaly $\Gamma \in \Lie{su}(2)\wedge \Lie{su}(2)$. Thus $U_{n}^{\epsilon}$ converges now to the solution of the SDE $dU_t = U_t \circ dB_t$ which \emph{must} be interpreted in the rough path sense since traditional stochastic calculus cannot take $\Gamma$ into account.

\subsection{Open questions}

The present result leads to some open questions both about the limit process with the area anomaly and about the discrete models which may converge to such a limit. 

It would be interesting to understand how the area anomaly fits in the Fock space description of Brownian Motion: the question is non-trivial because the Lévy area belongs to the second chaos and the presence of an area anomaly adds a zero-chaos component to the Lévy area.

Next, two-dimensional Brownian motion is known to exhibit conformal symmetry and it is natural to ask how our limit process behaves under conformal transformations.

Since we focus on the area drift, an important question that arises is: can Girsanov's theorem be extended to cancel the area anomaly by a change of measure on the rough path space?

In the introduction, it has been mentioned that other models than those we study, like the ones in a random environment from \cite{artBaur} or the ones on a lattice graph from \cite{artRWCristalLatt}, might generate a non-zero limit drift on the second level of the rough path. Moreover, there is a detail indicating that this drift might be the analogue of our area anomaly: the drift we remove from each of them before studying the convergence of the process is the analogue of the drift $nv$ from the present article. It would thus be interesting to see if we can get a convergence result for the models of these two articles in rough path topology and and what the stochastic area limit drift looks like in this case.

Our proof exhibits striking similarities with \cite{ABT}. Our "internal" $G_0$ space seems to play the same role as the compact sphere in their paper and their proof also uses theory of rough paths to control convergences. We would like to know if one can build models on Riemannian manifolds which may exhibit area anomalies. A good hint might be given by the construction from \cite{artRWCristalLatt}.

In \cite{bookSubLaplacians}, the main result from \cite{artRWDiscreteGroups} is generalized to sub-Laplacians with drift (i.e. we are here in a continuous setting). Just like in \cite{artRWDiscreteGroups}, the author gets a Berry-Esseen-like estimate of the heat kernel. The drift of the Laplacian is not automatically linked with the area anomaly and is considered here to be more like an additional problem than a central object that should be studied. We may use these results for studying our limit motion, i.e. the Brownian rough path with area anomaly.

We can also ask ourselves whether the present paper can lead to generalizing or improving the result from \cite{artCubature}.

Finally, Brownian motion belongs to the larger family of Lévy processes and, consequently, we may wonder whether Lévy processes may be also enhanced with area components and approximated by suitable discrete processes.

\bibliographystyle{plain}
\bibliography{biblio_article}

\end{document}